\documentclass[12pt]{article}
\usepackage{amsmath,amssymb,amsbsy,amsfonts,amsthm,latexsym,amsopn,amstext,
                                               amsxtra,euscript,amscd}
                                               
\begin{document}

\newfont{\teneufm}{eufm10}
\newfont{\seveneufm}{eufm7}
\newfont{\fiveeufm}{eufm5}
%
%
\newfam\eufmfam
                \textfont\eufmfam=\teneufm \scriptfont\eufmfam=\seveneufm
                \scriptscriptfont\eufmfam=\fiveeufm
%
%

\def\frak#1{{\fam\eufmfam\relax#1}}

\def\ts{\thinspace}

\newtheorem{theorem}{Theorem}
\newtheorem{lemma}[theorem]{Lemma}
\newtheorem{claim}[theorem]{Claim}
\newtheorem{cor}[theorem]{Corollary}
\newtheorem{prop}[theorem]{Proposition}
\newtheorem{definition}{Definition}
\newtheorem{question}[theorem]{Open Question}

\def\squareforqed{\hbox{\rlap{$\sqcap$}$\sqcup$}}
\def\qed{\ifmmode\squareforqed\else{\unskip\nobreak\hfil
\penalty50\hskip1em\null\nobreak\hfil\squareforqed
\parfillskip=0pt\finalhyphendemerits=0\endgraf}\fi}

\def\cA{{\mathcal A}}
\def\cB{{\mathcal B}}
\def\cC{{\mathcal C}}
\def\cD{{\mathcal D}}
\def\cE{{\mathcal E}}
\def\cF{{\mathcal F}}
\def\cG{{\mathcal G}}
\def\cH{{\mathcal H}}
\def\cI{{\mathcal I}}
\def\cJ{{\mathcal J}}
\def\cK{{\mathcal K}}
\def\cL{{\mathcal L}}
\def\cM{{\mathcal M}}
\def\cN{{\mathcal N}}
\def\cO{{\mathcal O}}
\def\cP{{\mathcal P}}
\def\cQ{{\mathcal Q}}
\def\cR{{\mathcal R}}
\def\cS{{\mathcal S}}
\def\cT{{\mathcal T}}
\def\cU{{\mathcal U}}
\def\cV{{\mathcal V}}
\def\cW{{\mathcal W}}
\def\cX{{\mathcal X}}
\def\cY{{\mathcal Y}}
\def\cZ{{\mathcal Z}}

\newcommand{\comm}[1]{\marginpar{%
\vskip-\baselineskip 
\raggedright\footnotesize
\itshape\hrule\smallskip#1\par\smallskip\hrule}}




\newcommand{\ignore}[1]{}

\def\vec#1{\mathbf{#1}}
\def\rem {\mathrm{ \ rem\,}}
\def\e{\mathbf{e}}
\def\vp{\varphi}
\def\loN{\mathbb{N}_0}
\def\sh{\textstyle{\frac{1}{2}}}

\def\longlr{\longleftrightarrow}
\def\longra{\longrightarrow}
\def\lra{\leftrightarrow} 
\def\Lra{\Leftrightarrow} 
\def\ra{\rightarrow}
\def\beq{\backsimeq}
\def\dar{\downarrow}
\def\uar{\uparrow}
\def\rem{\rm{rem}}
\def\ded{\vdash}
\def\0P{\sP_0}
\def\1P{\sP_1}
\def\sb{\textsubscript}
\def\sp{\textsuperscript}
\def\bdot{\times}
\def\dom{\mathrm{dom}}
\def\ran{\mathrm{ran}}
\def\back{\hspace{-0.8em}\space}



\def\A{\mathbb{A}}
\def\B{\mathbf{B}}
\def \C{\mathbb{C}}
\def \F{\mathbb{F}}
\def \K{\mathbb{K}}
\def \Z{\mathbb{Z}}
\def \R{\mathbb{R}}
\def \Q{\mathbb{Q}}
\def \N{\mathbb{N}}
\def\Z{\mathbb{Z}}

\def\z{\zeta}

\setlength{\textheight}{43pc}
\setlength{\textwidth}{28pc}

\def\li {\mathrm{li}\,}
\def\lam{\lambda}
\def\N{\mathbb{N}}
\def\s{\sigma}
\def\g{\gamma}
\def\a{\alpha}
\def\b{\beta}
\def\e{\epsilon}
\def\d{\delta}
\def\D{\Delta}
\def\vt{\vartheta}

\def\step#1{{\bf (#1)}}


\title{The Decidability of the Riemann Hypothesis}

\author{
{Kevin Broughan} \\
{ Department of Mathematics and Statistics}\\
{University of Waikato}\\
{Hamilton, New Zealand}\\
{\tt kevin.broughan@waikato.ac.nz} \\
}
\date{\today}


\maketitle

\begin{abstract}
Using first order predicate logic and results of the complex analysis of Takeuti which is based on a type theory and the work of Kreisel, and which gives a conservative extension of first order Peano arithmetic (PA), it is shown, assuming all critical zeros of the Riemann zeta function are simple, that the Riemann hypothesis is decidable in PA. 
\end{abstract}


{\bf MSC2020 Subject Classification}: 11M26


\section{Introduction}
\label{int}

It is apparent that many mathematical propositions are undecidable. See for example \cite[Chapter 10]{bro}. After over 160 years of strenuous endeavours by some of the best mathematical minds of many generations but without a resolution of the Riemann hypothesis (RH), one might suspect that it is undecidable also. However, most who have thought about the issue believe it can be resolved, and there has been recent significant progress in that direction. See for example the work of Nicolas, Rogers and Tao, Dobner, Polymath15, Ono et al and Gonek and Bagchi, recorded for example in \cite{bro}.\par

In \cite[Chapter 11]{bro} there are some proof ideas on showing RH is decidable. However, the assumption that all critical line zeros are simple was made. It is the purpose of this note to remove that assumption.\par

To do this two powerful theories are used: recursive function theory and Takeuti's extension of first order Peano arithmetic which is based on the work of Kreisel and on a form of type theory. This extension not only provides a great deal of real mathematical analysis and complex analysis, but also has the property of being ``conservative". That means
anything that can be proved using its theory can also be proved using Peano arithmetic.\par

The result from recursive function theory is much simpler to state and understand. It is in stark contrast to the theorem of Rice, which is often used to demonstrate the ubiquity of undecidable propositions.\par

In Section \ref{defs} there are the main definitions,
in Section \ref{lems} the supporting lemmas, including the statement of the result from recursive function theory and some of the results from Takeuti's system for real and complex analysis (PAT), and in PAT an increasing enumeration of the distinct critical zeros of $\z(s)$ which have positive imaginary part is derived. Finally, in Section \ref{res} the proof of the main result is given.\par

It goes without saying that a proof of the decidability of RH implies all of the equivalent forms of the hypothesis are decidable, including those reported in \cite{bro12, bro}. In addition, the simplicity of the proof indicates that extensions to families of L-functions such as the extended Selberg class should not be too difficult.\par

Here is an overview of the sense in which RH is decidable. We are assuming all critical zeros are simple. First the distinct zeros on the critical line with positive imaginary part are enumerated $\g_1<\g_2<\cdots$. Then the positive critical strip is subdivided into (if needed indented) rectangular neighbourhoods $\D_j$ of the $\g_j$. Then RH is decidable in the sense that the subset, $S$ say, of $\N$ of the indices $j$ such that $\D_j$ contains an off critical line zero of $\z(s)$ is a decidable subset of $\N_0$ (and thus its complement the set of $j$ such that $\D_j$ contains just $\g_j$ as its only distinct zero is also a decidable subset). In addition RH is equivalent to the statement $S=\emptyset$ which, as shown in the proof of Theorem \ref{thm:RHdec}, is decidable.\par

We say here that a subset $S$ of $\N_0$ is {\bf decidable} if there exists a partial recursive function $f$, defined everywhere on $\N_0$ such that if $n\in S$ then $f(n)=1$, otherwize $f(n)=0$. This sense for decidability is comparable with that used for the theorem of Davis, Putnam, Robinson and Matiysevich \cite{DPRM, Mati} or \cite[theorem 10.15]{bro}: there does not exist a single algorithm which would determine whether any polynomial with integer coefficients has an integer solution.\par

We also use the {\bf Church-Turing thesis}, which asserts the existence of an identification of each partial recursive function with an algorithm (defined in first order PA or a conservative extension), and vice versa.  

This work is based principally on the methods developed by Gaisi Takeuti, which follows that of Gehard Gentzen and George Kreisel. There is a summary of some of this work, especially that of Gentzen and Takeuti, in \cite{bro}. See Sections M.11.6, O.3 and O.4. To clarify confusion, Takeuti does not base his development of real and complex analysis on the form of arithmetic shown by Gentzen to be consistent, but simply first order Peano arithmetic PA.
\bigskip

\section{Definitions}
\label{defs}

In this section the definitions which are used in the note are set out. These are from recursive function theory, Takeuti's elementary complex analysis from mathematical logic, and from analytic number theory. There are no derivations and the reader might need to consult the references.\par

\subsection{Recursive functions}

This is a study of a class of functions $f:\loN\ra\loN$ where $\loN=\{0,1,2,3,\ldots\}$ defined procedurally, and where the domain is not necessarily the whole of $\loN$ but is determined by the structure of $f$. See \cite[Sections N.2-N.4]{bro} and the references given there, especially \cite{cut}.\par

 The class of {\bf partial recursive functions}, denoted PRF, is the least class of functions $f\colon \loN^k\ra\loN$ for some $k\ge 1$, where the domain $A$ of $f$ need not be all of $\loN^k$, but is determined by the structure of the function's definition, such that the class is closed under the operations (a), (b) and (c), where

(a) PRF contains the zero function $Z(x)=0$ for all $x\in\loN$, the successor function $S(x)=x'=x+1$ for all $x\in \loN$, and for each $j$ with $1\le j\le k$ the projection function $\pi_j^k(x_1,\ldots,x_k)=x_j$.
These three functions are everywhere defined so are {\bf total}.
\bigskip

(b) If $f(\vec{x})$ and $g_1(\vec{y}),\ldots,g_k(\vec{y})$ are in PRF, so is the {\bf composite function}
$h(\vec{y})=f(g_1(\vec{y}),\ldots,g_k(\vec{y}))$.
If any of the $g_j$ are undefined at a particular $\vec{y}$ then so is $f(\vec{y})$.
\bigskip

(c) PRF is closed under {\bf primitive recursion}, namely if $f(\vec{x})$ and $g(\vec{x},y,z)$ are functions in PRF then so is $h(\vec{x},y)$ which is defined by
\begin{align*}
h(\vec{x},0)&=f(\vec{x}),\\
h(\vec{x},y+1)&= \begin{cases}
   & g(\vec{x},y,h(\vec{x},y)),\\
&\uparrow \text{ if $h(\vec{x},y)$ is undefined}.\\
 \end{cases}\\
\end{align*}
The symbol $\uparrow$ means undefined and $\downarrow$ defined.
\bigskip

(d) The set of functions PRF is closed under the operation {\bf minimalization}, namely, if $g(\vec{x},y)$ is in PRF then so is $h(\vec{x})=\mu_y(g(\vec{x},y)=0)$ where
\begin{align*}
&h(\vec{x})\text{ is the least $y$ such that $g(\vec{x},y)=0$},\\
&~~~~~\text{and, for $0\le j<y,~ g(\vec{x},j)$, is defined and non-zero}\\
&h(\vec{x}) \text{ is undefined if no $y$ exists with $g(\vec{x},y)=0$},\\
&~~~~~\text{ or, if such a $y$ exists, there is a $j$ with $1\le j<y$ with $g(\vec{x},j)$ undefined}.\\
\end{align*}

The functions in PRF which are everywhere defined are called total functions and the corresponding subset is denoted by TPR.\par

The {\bf Church--Turing thesis} is used to simplify arguments. A statement of this is that a partial function $f:\loN^k\ra\loN$ is in PRF if and only if there is an algorithm which computes $f$. This is in the sense that the domain of $f$ is all points where $f$ returns a value and the complement of the domain is where the algorithm which computes the function does not halt.\par

 Let $A\subset \loN^k$ be a subset. The subset $A$ is called {\bf decidable} if its characteristic function $C_A$ is a total partial recursive function TPR.

A subset $A\subset \loN$ (or $\loN^k$) is {\bf recursively enumerable} (denoted r.e.) if it is the domain of a partial recursive function PRF. This is somewhat unintuitive, but the meaning is clearer in the case $k=1$ where
for a subset $A\subset \loN$, the following are equivalent: (a) $A$ is recursively enumerable, and
(b) $A$ is the empty set or the range of a total partial recursive function TPR on $\loN$ \cite[Theorem N.4]{bro}.
\bigskip

\subsection{Mathematical logic}

Very little is needed from this wide field of ideas and techniques. For a concise summary of the basic concepts see for example \cite[Appendix M]{bro}, especially Section M.11.1 for first order Peano arithmetic PA, and Section M.11.6 for Takeuti's conservative extension of PA for real and complex analysis which in \cite{bro} is called PAT (but Takeuti calls ``elementary complex analysis"). Consulting Takeuti's well written \cite{tak} is worthwhile, especially Chapters 3 and 4. Motivation might be enhanced by commencing reading at the final section Section 4.8, where after a description of the sort of results which cannot be proved in PAT such as those relying on the axiom of choice, there is the remark
\begin{quote}
``...classical analytic number theory, for example, the theory in [2] (which is Ingham's 1971 text, {\it The distribution of prime numbers}) can be carried out in elementary complex analysis."
\end{quote}

An extension $T_2$ of a theory $T_1$ is called {\bf conservative} if any proposition which can be proved in the theory
$T_2$ can be proved in $T_1$.\par

\subsection{Analytic number theory}

Very little is needed from Analytic number theory, only the two real valued Riemann-Siegel functions $Z$ and $\vt$ which are defined as follows. (See for example \cite[Section 6.5]{edwards} or \cite[Section 4.17]{tit}.)\par

Take $s=\sh+it$ and $x=y=\sqrt{t/(2\pi)}$ in the approximate functional equation for $\z(s)$, (see for example \cite[Theorem G.2]{bro}), to get
$$
\z(\sh+it)= \sum\limits_{n\le x} \frac{1}{n^{\sh+it}}+\chi(\s+it)\sum\limits_{n\le x} \frac{1}{n^{-\sh+it}}+O\left(\frac{1}{t^{1/4}}\right),
$$
where $|\chi(\sh+it)|=1$. Define 
$$
\vt(t)=-\frac{\arg(\chi(\sh+it))}{2} \text{ and } Z(t)=e^{i\vt(t)}\z(\sh+it).
$$
We get 
$$
Z(t)=2\sum_{n\le x} \frac{\cos(\vt(t)-t\log n)}{\sqrt{n}}+O\left(\frac{1}{t^{1/4}}\right), 
$$
and can write
$$
\vt(t)= \frac{1}{2}\log\left(\frac{t}{2\pi}\right)-\frac{t}{2}-\frac{\pi}{8}+O\left(\frac{1}{t}\right).
$$

\bigskip

\section{Lemmas}
\label{lems}

The first lemma is from Takeuti's real analysis FA \cite[Chapter 3]{tak}:.
\begin{lemma}{\bf (Takeuti)}\cite[Proposition 3.5.1]{tak}
\label{tak3.5.1}
The supremum of a real sequence which is bounded above exists as a real number in PAT. The corresponding property of the infimum can be stated in a similar manner.
\end{lemma}

\begin{lemma}\cite[Theorem N.10]{bro}\cite[Theorem 7-2.14]{cut}
\label{thm:cut7-2.14}
A subset of the natural numbers $\loN$ is total partial recursive, and hence decidable, if and only if it is the range of a strictly increasing total partial recursive function, that is to say if it can be enumerated TPR in a strictly increasing order.
\end{lemma}

This next lemma is a variation on \cite[Section 11.10 (3)]{bro}. All of the arguments can be carried out in either Takeuti's  conservative extensions of PA called FA for mathematical analysis or his ``elementary complex analysis" for complex analysis, i.e. in what we call PAT.\par

\begin{lemma}\label{critical-zeros} {\bf (Enumeration of the critical zeta zeros)}
The critical line zeros of $\z(s)$ with positive imaginary part are recursively enumerable TPR in increasing order.
\end{lemma}
\begin{proof} 
\step{1}
By Lemma \ref{tak3.5.1} every real sequence which is bounded above has a real least upper bound. This implies the Archimedian axiom for Takeuti's system FA holds, and thus for each pair of reals $x,y$ with $x<y$ there is a rational number $r$ say with $x<r<y$. By \cite[Theorem 4.5.2]{tak}, the zeros of an entire function are isolated. Thus the zeros of $\z(s)$ on the critical line are isolated in PAT, and therefore, since 
$$Z(t)=exp(i\vt(t))\z(\sh+it),~t\in\R$$
so are those of $Z(t)$ in $\R$. \par

If $\g$ is such a zero and $x_\g<\g<y_\g$ is an interval in $\R$ which contains the zeta zero with imaginary part $\g$ and no other zero, let $r_\g$ be a rational with $x_\g<r_\g<y_\g$. This shows that the zero's of $Z(t)$ are an enumerable set.\par

By the results of Titchmarsh \cite{tit1934}, based on the Approximate Functional equation \cite[Theorem G.2]{bro}, the positive zeros of $Z(t)$ have an infinite subset $h_n$ with $h_n \ra \infty$.\par

Now let $\g_1$ be the real infimum of the set of zeros of $Z(t),~t>0$. Since the zeros are isolated $Z(\g_1)=0$. If $\g_n$ has been defined, let $\d>0$ be such that $Z(t)^2>0$ on $(\g_n,\g_n+\d]$ and let $(a_m)$ be an enumeration of the zeros of $Z(t), ~t>0$ which satisfy $a_m>\d$. Let $\g_{n+1}=\inf_m a_m$. Then $Z(\g_{n+1})=0$ and $Z(t)^2>0$ on $(\g_n,\g_{n+1})$. All this is true in FA and thus in PAT.\par
\bigskip

\step{2} Let $S:= \{\g_n:n\in\N\}$ and $T=\{\g:Z(\g)=0\}$ so $S\subset T$. If $S\neq T$ let $\g\in T\setminus S$. Note that for all $n$, $\g_n<\g<\g_{n+1}$ is false, since by Step (1) we have $Z(t)^2>0$ on $(\g_n,\g_{n+1})$. Hence $\g_n<\g$ for all $n$. By Lemma \ref{tak3.5.1}, the supremum $\g_0:= \sup_n \g_n$ exists as a real in FA. It follows by continuity that $Z(\g_0)=0$. But zeros are isolated so  $\g_0=\g_n$ for some $n$, which is false. Hence $S=T$ and the set of positive zeros of $Z(t)$ can be written
$$
\{\g_1<\g_2<\g_3<\cdots\}\text{ with }\g_n\ra \infty.
$$
This completes the proof.
\end{proof}

\begin{lemma}{\bf (Takeuti)}\cite[Theorem 4.5.6]{tak}
\label{tak4.5.6}
Let $f$ be holomorphic in a circular disk $\D$ and be not identically zero in $\D$. Let $z_j$ be the zeros of $f(z)$, each zero being counted as many times as its multiplicity. For every closed curve $\g$ in $\D$ which does not pass through a zero we have 
$$
\sum_j n(\g,z_j)= \frac{1}{2\pi i} \oint_\g \frac{f'(z)}{f(z)}~dz,
$$
where the sum has only a finite number of non-zero terms.
\end{lemma}

\section{Result}
\label{res}

Again, Takeuti's conservative extension of Peano Arithmetic PAT is used (see Appendix M Section M.11.6). It is shown that RH is decidable in PAT and therefore in PA.\par

\begin{theorem}\label{thm:RHdec}
Assuming first order Peano arithmetic PA, the Riemann hypothesis is decidable.
\end{theorem}
\begin{proof}

\step{1} Let $(\g_n:n\in\loN)$ be the increasing enumeration of the zeros given by Lemma \ref{critical-zeros}. For $n,j\in\loN$, let $\C_{n,j}$ be a circular contour center $\g_n$ and radius $1/j$. Then by Lemma \ref{tak4.5.6} the integral
$$
\frac{1}{2\pi i}\oint_{C_{n,j}} \frac{\xi'(z)}{\xi(z)}~dz
$$
exists and is monotonically decreasing in $j$ with values in $\mathbb{N}$. By Lemma \ref{tak3.5.1}, the limit
$$
\lim_{j\rightarrow \infty} \frac{1}{2\pi i}\oint_{C_{n,j}} \frac{\xi'(z)}{\xi(z)}~dz
$$
exists in PAT. It is the multiplicity of the zero $\g_n$. Let its value be denoted by the positive integer $l_n$.
\bigskip

\step{2} Let $f(0)=\g_0=1$, and using Lemma \ref{critical-zeros}, for all $n\ge 1$ let $f(n)=\g_n$ be the strictly increasing enumeration of the imaginary parts of the zeros of $\z(s)$ on the critical line. Let $J_0:= [i,1+i]$ and for each $n\ge 1$ set
$$
\d_n:= \frac{f(n)+f(n-1)}{2}\text{ and } J_n:= [i\d_n,1+i\d_n].
$$
Then there is no zero of $\z(s)$ in the intervals $J_0$ or $J_1$.
For each zero of $\z(s)$ in the interval $J_n$, indent the contour with a semi-circular path downwards, so the new zeros are associated with the $n$th region, which is the interior of the indented $\D_n$ defined next.
\bigskip

\step{3} Now set $g(0)=0$ and $h(0)=0$. For $n\in\N$ define the rectangular contour
$$
\D_n:=[i\d_n,1+i\d_n]\cup [1+i\d_n,1+i\d_{n+1}]\cup [1+i\d_{n+1},i\d_{n+1}]\cup [i\d_{n+1},i\d_n],
$$
indented as may be needed, and set
$$
M(n):= \frac{1}{2\pi i}\oint_{\D_n} \frac{\xi'(z)}{\xi(z)}~dz.
$$
Recall $l_n$ is defined in Step (1). Let
$$
g(n):= \begin{cases}
g(n-1) & \text{ if } M(n)=l_n,\\
n &\text{ otherwize},\\
\end{cases}
$$
and
$$
h(n):= \begin{cases}
h(n-1) & \text{ if } M(n)=l_n,\\
1 &\text{ otherwize}.\\
\end{cases}
$$
Then $g$ and $h$ are non-decreasing function which, by PAT and the Church--Turing thesis, are TPR. By Lemma \ref{thm:cut7-2.14}, their ranges are decidable sets. But the range of $g$, $S$ say indexes all of the (possibly indented) rectangles containing a zero of $\z(s)$ off the critical line and the range of $h$ is $\{0\}$ if and only if $S=\emptyset$ if and only if RH is true. Thus, in the sense set out in the penultimate paragraph of the introduction, RH is decidable.\par
\end{proof}

\end{document}